\documentclass[12pt]{amsart}
\usepackage{amssymb,amsmath,amsfonts,pictex,graphicx,ifthen,fullpage,etex}

\DeclareMathOperator{\id}{id}
\DeclareMathOperator{\intr}{int}

\DeclareMathOperator{\Itin}{Itin}

\DeclareMathOperator{\Min}{Min}

\newcommand{\HE}{\stackrel{\textrm{HE}}{\simeq}}

\newcommand{\N}{\mathcal{N}}

\newcommand{\Z}{\mathbb{Z}}
\newcommand{\R}{\mathbb{R}}

\newtheorem{prop}{Proposition}[section]
\newtheorem{fact}[prop]{Fact}
\newtheorem{lemma}[prop]{Lemma}
\newtheorem{corollary}[prop]{Corollary}

\newtheorem{defn}[prop]{Definition}

\newtheorem*{ack}{Acknowledgements}
\newtheorem*{claim}{Claim}

\newtheorem*{Thm1}{Theorem 1}

\newtheorem*{ThmA}{Theorem A}
\newtheorem*{ThmB}{Theorem B}
\newtheorem*{ThmC}{Theorem C}
\newtheorem*{ThmD}{Theorem D}
\newtheorem*{ThmE}{Theorem E}
\newtheorem*{ThmF}{Theorem F}

\newtheorem*{ThmA'}{Theorem A$'$}
\newtheorem*{ThmB'}{Theorem B$'$}
\newtheorem*{ThmC'}{Theorem C$'$}
\newtheorem*{ThmD'}{Theorem D$'$}
\newtheorem*{ThmE'}{Theorem E$'$}
\newtheorem*{ThmF'}{Theorem F$'$}

\begin{document}

                %%%%%%%%%%%%%%%%%%%%%%%%%%
                %                        %
%%%%%%%%%%%%%%%%%    Title & Abstract    %%%%%%%%%%%%%%%%%
                %                        %
                %%%%%%%%%%%%%%%%%%%%%%%%%%

\title[Generalizing the Croke-Kleiner Construction]{Generalizing the Croke-Kleiner Construction}
\begin{abstract}
It is well known that every word hyperbolic group has a well-defined visual boundary.
An example of C. Croke and B. Kleiner shows that the same cannot be said for CAT(0) groups.
All boundaries of a CAT(0) group are, however, \textit{shape equivalent}, as observed
by M. Bestvina and R. Geoghegan.  Bestvina has asked if they also satisfy
the stronger condition of being \textit{cell-like equivalent}.
This article describes a construction which will produce
CAT(0) groups with multiple boundaries.  These groups have very complicated boundaries in high
dimensions.  It is our hope that their study may provide insight into Bestvina's question.
\end{abstract}

                %%%%%%%%%%%%%%%%%%%%%%%%%%
                %                        %
%%%%%%%%%%%%%%%%%    Author & Address    %%%%%%%%%%%%%%%%%
                %                        %
                %%%%%%%%%%%%%%%%%%%%%%%%%%

\author{Christopher Mooney}
\address{Department of Mathematical Sciences, University of Wisconsin-Milwaukee, Milwaukee, Wisconsin 53201; (262) 210-0201}
\email{cpmooney@uwm.edu}
\date{\today}
\maketitle

                %%%%%%%%%%%%%%%%%%%%%%%%%%
                %                        %
%%%%%%%%%%%%%%%%%    The Good Stuff !    %%%%%%%%%%%%%%%%%
                %                        %
                %%%%%%%%%%%%%%%%%%%%%%%%%%

\section{Introduction}
\label{sec:intro}
The CAT(0) condition is a geometric notion of nonpositive curvature, similar to the
definition of Gromov $\delta$-hyperbolicity.  A complete geodesic space $X$ is called CAT(0)
if it has the property that geodesic triangles in $X$ are ``no fatter'' than geodesic
triangles in euclidean space (see \cite[Ch II.1]{BH} for a precise definition).
The \textit{visual} or \textit{ideal boundary} of $X$, denoted \(\partial X\),
is the collection of endpoints of geodesic rays
emanating from a chosen basepoint.
It is well-known that \(\partial X\) is well-defined and independent of choice of basepoint.
Furthermore, when given the cone topology, \(X\cup\partial X\) is a $Z$-set compactification for $X$.
A group $G$ is called CAT(0) if it acts geometrically
(i.e. properly discontinuously and cocompactly by isometries) on some CAT(0) space $X$.
In this setup, we call $X$ a cocompact CAT(0) $G$-space
and \(\partial X\) a CAT(0) boundary of $G$.\\

It is an important fact in geometric group theory that every negatively curved group
(that is, every word hyperbolic group) has a well-defined visual boundary.  Specifically,
if a group $G$ acts geometrically on two different CAT(-1) spaces, then the visual boundaries of these
spaces will be homeomorphic.  In the absence of strict negative curvature, the situation
becomes more complicated.\\

We will call a CAT(0) group \textit{rigid} if it has only one
topologically distinct boundary.
P.L. Bowers and K. Ruane showed that if G splits as the product of a
negatively curved group with a free abelian group, then G is rigid (\cite{BR}). Ruane proved
later in \cite{Ru} that if G splits as a product of two negatively curved groups, then G is rigid.
T. Hosaka has extended this work to show that in fact it suffices to know that G splits
as a product of rigid groups (\cite{Ho}). Another condition which guarantees rigidity is knowing
that G acts on a CAT(0) space with isolated flats, which was proven by C. Hruska in \cite{Hr}.\\

Not all CAT(0) groups are rigid, however: C. Croke and B. Kleiner constructed in \cite{CK}
an example of a non-rigid CAT(0) group $G$.  Specifically, they showed that $G$ acts on two different
CAT(0) spaces whose boundaries admit no homeomorphism.
J. Wilson proved in \cite{Wi} that this same group has \textit{uncountably many} boundaries.
More recently it has been shown in \cite{MoKnot} that the knot group $G$ of any connected sum of two non-trivial
torus knots has uncountably many CAT(0) boundaries.\\

On the other end of the spectrum, it has been observed by M. Bestvina (\cite{Be}), R. Geoghegan
(\cite{Ge}), and P. Ontaneda (\cite{On}) that all boundaries of a given CAT(0) group are \textit{shape equivalent}.
Bestvina then posed the question of whether they satisfy the stronger condition of being
\textit{cell-like equivalent}.
R. Ancel, C. Guilbault, and J. Wilson showed in \cite{AGW}
that all the currently known boundaries of Croke and Kleiner's original group satisfy
this property; they are all cell-like equivalent to the Hawaiian earring.\\

Further progress on Bestvina's question has been hampered by a lack of examples of non-rigid CAT(0) groups.
The results in \cite{MoKnot} have boundaries which are similar to the boundaries of the Croke-Kleiner group
and, as such, are unlikely to shed new light on Bestvina's question.  One simple approach to producing
other non-rigid CAT(0) groups is to take a direct product \(G\times H\) of two CAT(0) groups where one
of the factors is not rigid (either the Croke-Kleiner group or one of the groups from \cite{MoKnot}).
However, it was proven in \cite{MoCE} that the answer to Bestvina's question is ``Yes'' for CAT(0) groups
of this form.\\

The goal of this article is to describe a construction which will yield a richer collection of non-rigid CAT(0)
groups.  Our work borrows freely from the main ideas of \cite{CK}, but in the end, we have a flexible strategy
for producing CAT(0) groups which have very complicated boundaries in high dimensions.
We are not claiming new progress on Bestvina's question, but it is our hope that the study of this new collection
may provide insight.\\

The main theorem of this paper is the following.\\

\begin{Thm1}
\label{thm:generalCK}
Let $\Gamma_-$ and $\Gamma_+$ be infinite CAT(0) groups and $m$ and $n$ be positive integers.
Then the free product with amalgamation
\[
	G=\bigl(\Gamma_-\times\Z^m\bigr)\ast_{\Z^m}\bigl(\Z^m\times\Z^n\bigr)\ast_{\Z^n}\bigl(\Z^n\times\Gamma_+\bigr)
\]
is a non-rigid CAT(0) group.
\end{Thm1}

Observe that we get the Croke-Kleiner group if we take \(\Gamma_-=\Gamma_+=\Z\) and \(m=n=1\) in this theorem.
In general, though, any boundary of $G$ will contain spheres of dimension \(m+n-1\).  Thus by choosing $m$ and $n$
to be large, we get a non-rigid CAT(0) group whose boundaries have high dimension.\\

\section{Croke and Kleiner's Original Construction}
	\label{sec:origCK}
Before diving into the proof of Theorem \ref{thm:generalCK},
we quickly sketch the proof of the main theorem of \cite{CK}.  The CAT(0) spaces $X$ constructed have the property
that each is covered by a collection of closed convex subspaces, called \textit{blocks}.  The visual
boundary $\partial B$ of every block $B$ is the suspension of a Cantor set.  The suspension points
are called \textit{poles}.  The intersection of two blocks is a Euclidean plane
called a \textit{wall}.  We then have
the following five statements for each $X$\\

\begin{ThmA} \cite[Sec 1.4]{CK}
The nerve $N$ of the collection of blocks is a tree.
\end{ThmA}

\begin{ThmB} \cite[Le 3]{CK}
Let \(B_0\) and \(B_1\) be blocks, and $D$ be the distance between the corresponding vertices in
$N$.  Then:\\
\begin{enumerate}
\item If \(D=1\), then \(\partial B_0\cap\partial B_1=\partial W\) where $W$ is the wall \(B_0\cap B_1\).\\
\item If \(D=2\), then \(\partial B_0\cap\partial B_1\) is the set of poles of \(B_{\frac12}\) where
\(B_{\frac12}\) intersects \(B_0\) and \(B_1\).\\
\item If \(D>2\), then \(\partial B_0\cap\partial B_1=\emptyset\).\\
\end{enumerate}
\end{ThmB}

A \textit{local path component} of a point in a space is a path component of an open neighborhood of that point.\\

\begin{ThmC} \cite[Le 4]{CK}
Let $B$ be a block and \(\zeta\in\partial B\) not be a pole of any neighboring block.  Then $\zeta$ has a local
path component which stays in \(\partial B\).
\end{ThmC}

\begin{ThmD} \cite[Co 8]{CK}
The union of block boundaries in $\partial X$ is the unique dense \textit{safe path} component
of \(\partial X\).
\end{ThmD}

The definition of \textit{safe path} will be given in Section \ref{sec:spcomps}.
For now, it suffices to understand that
Theorem D gives a way to topologically distinguish the union of block boundaries in $\partial X$.  With these
thereoms in hand, it is not hard to prove that given two constructions $X_1$ and $X_2$, any homeomorphism
\(\partial X_1\to\partial X_2\) takes poles to poles, block boundaries to block boundaries, and wall boundaries
to wall boundaries.  The last piece of the puzzle is Theorem E.  Given \(0<\theta\le\pi/2\), we can construct
$X_{\theta}$ in such a way that the minimum distance between poles is $\theta$.
This distance is in the
sense of the Tits path metric on the boundary of a block containing both poles.
For a block $B$, we denote by $\Pi B$
the set of poles of blocks which intersect $B$ at a wall.\\

\begin{ThmE} \cite[Le 9]{CK} (also \cite[Prop 2.2]{Wi})
For a block $B$, the union of boundaries of walls of $B$ is dense in \(\partial B\) and
\(\overline{\Pi B}\) is precisely the set of points of \(\partial B\) which are a Tits distance
of $\theta$ from a pole of $B$.
\end{ThmE}

Using these five theorems, we get the following statement.\\

\begin{ThmF}
Let $B$ be a block and $L$ be a suspension arc of \(\partial B\).  Then \(\bigl|L\cap\overline{\Pi B}\bigr|=1\) iff
\(\theta=\pi/2\).
\end{ThmF}

Therefore \(\partial X_{\pi/2}\not\approx\partial X_\theta\) for any \(\theta<\pi/2\), which is the main theorem
of \cite{CK}.  In this article, we produce for every group $G$ in question a pair of cocompact CAT(0) $G$-spaces $X$ and $X'$
and show that Theorems A-D still hold.  These four theorems, along with an analogue to Theorem E will
be used to prove that there is no homeomorphism \(\partial X\to\partial X'\).

\section{Block Structures on CAT(0) Spaces}
	\label{sec:NervesandItineraries}
We begin by observing that the work in Sections 1.4-5 of \cite{CK} does not depend on the specific
construction used in in \cite{CK}.  The same observations apply if we replace their definition of a
\textit{block} with the following one.\\

\begin{defn}
\label{defn:blockstructure}
Let $X$ be a CAT(0) space and \(\mathcal{B}\) be a collection of closed convex subspaces covering $X$.
We call $\mathcal B$ a \textit{block structure on $X$} and its elements \textit{blocks} if
\(\mathcal{B}\) satisfies the following three properties:\\
\begin{enumerate}
\item Every block intersects at least two other blocks.\\
\item Every block has a $(+)$ or $(-)$ parity such that two blocks intersect only if they have opposite parity.\\
\item There is an \(\epsilon>0\) such that two blocks intersect iff their $\epsilon$-neighborhoods intersect.\\
\end{enumerate}
\end{defn}

If we refer to blocks as \textit{left} or \textit{right}, we mean that the former have parity $(-)$ and the
latter have parity $(+)$.
The \textit{nerve} of a collection \(\mathcal{C}\) of sets is
the (abstract) simplicial complex with vertex set \(\{v_B|B\in\mathcal{C}\}\)
such that a simplex \(\{v_{B_1},...,v_{B_n}\}\) is included whenever
\(\bigcap_{i=1}^nB_i\neq\emptyset\).
In exactly the same way as in \cite{CK}, we get that the nerve $\N$ of the collection of blocks
is a tree, and we can define the itinerary of a geodesic.
A geodesic $\alpha$ is said to \textit{enter} a block if it passes through a point which is not in any other block.
The itinerary of $\alpha$ is defined to be the list
\([B_1,B_2,...]\) where $B_i$ is the $i^{\textrm{th}}$ block that $\alpha$ enters.
This list is denoted by \(\Itin\alpha\).
The following lemma follows in exactly the same way as \cite[Le 2]{CK},
which simply uses the fact that a block $B$ is convex and that
its topological frontier is covered by the collection of blocks corresponding to the link
in $\N$ of the vertex $v_B$.\\

\begin{lemma}
\label{le:itingeo}
If \(\Itin\alpha=[B_1,B_2,...]\), then \([v_{B_1},v_{B_2},...]\) is a geodesic in $\N$.\\
\end{lemma}

We may also talk about the \textit{itinerary between two blocks}.
If \([v_{B_1},...,v_{B_n}]\) is the geodesic edge path in $\N$ connecting
two vertices \(v_{B'_0}\) and \(v_{B'_1}\), then we call \([B_1,...,B_n]\) the itinerary between
\(B'_0\) and \(B'_1\) and write
\[
	\Itin[B'_0,B'_1]=[B_1,...,B_n].
\]
The two notions of itineraries are related in the following way:
The itinerary of a geodesic segment $\alpha$ is the shortest itinerary \(\Itin[B'_0,B'_1]\) for which $\alpha$ begins in
$B'_0$ and ends in $B'_1$.  Note also that the same observations which gave us Lemma \ref{le:itingeo}
also provide the following:\\

\begin{lemma}
\label{le:itins}
Let $B'_0$ and $B'_1$ be blocks, write \(\Itin[B'_0,B'_1]=[B_1,...,B_n]\), and let $\alpha$ be a geodesic
beginning in $B'_0$ and ending in $B'_1$.
Then:
\begin{enumerate}
\item $\alpha$ enters $B_k$ for every \(1<k<n\).\\
\item $\alpha$ passes through \(B_k\cap B_{k+1}\) for every \(1\le k<n\).\\
\item \(\bigcup_{k=1}^nB_k\) is convex.\\
\end{enumerate}
\end{lemma}

We call a geodesic ray \textit{rational} if its itinerary is finite and
\textit{irrational} if its itinerary is infinite.  A point of \(\partial X\)
is called \textit{irrational} if it is the endpoint of an irrational geodesic
ray; otherwise we call it \textit{rational}.  We denote the set of rational points
of \(\partial X\) by $RX$ and the set of irrational points by $IX$.\\

\begin{lemma}
\label{le:irrationaldistance}
Let $\alpha$ be an irrational geodesic ray.
Then for any block \(B_0\),
\[
	\lim_{t\to\infty}d\bigl(\alpha(t),B_0\bigr)=\infty.
\]
\end{lemma}

\begin{proof}
Write \(\Itin\alpha=[B_1,B_2,...]\).  Since $\N$ is a tree, we can find \(M>1\)
so that for every \(m\ge M\), \(\Itin[B_0,B_m]\ni B_M\).
For $m\ge M$, choose a time $t_m$ such that \(\alpha(t_m)\in B_m\).
Then
\begin{align*}
	\lim_{t\to\infty}
	d\bigl(\alpha(t),B_0\bigr)
&\ge
	\lim_{t\to\infty}
	d\bigl(\alpha(t),B_M\bigr) \\
&=
	\lim_{m\to\infty}
	d\bigl(\alpha(t_m),B_M\bigr) \\
&\ge
	\lim_{m\to\infty}
	d\bigl(B_m,B_M\bigr).
\end{align*}
Hence, it suffices to prove the following.\\

\begin{claim}
Let $\epsilon$ be given as in condition (3) of Definition \ref{defn:blockstructure}.
Then whenever \(d(v_{B},v_{B'})\ge 2k\), we have \(d(B,B')\ge 2k\epsilon\).
\end{claim}
\indexspace

Note that whenever \(d(v_{B},v_{B'})=2\), then we have \(d(B,B')\ge 2\epsilon\)
because the $\epsilon$-neighborhoods of $B$ and $B'$ do not overlap.
Assume \(\Itin[B,B']=[B_0,B_1,...,B_n]\),
where \(n\ge 2k\).  Then for any \(x\in B\) and \(x'\in B'\),
the geodesic \([x,x']\) passes through
\(B_{2i}\) for \(0\le i\le k\) at some point $z_i$.
So
\begin{align*}
	d(x,x')
&=
	\sum_{i=0}^{k-1}d(z_i,z_{i+1}) \\
&\ge
	2k\epsilon.
\end{align*}
\end{proof}

\begin{corollary}
\label{co:RXIX}
\rule{0cm}{0cm}
\begin{enumerate}
\item \(RX\) is the union of block boundaries in \(\partial X\), and $IX$ is its complement.\\
\item If \(\zeta\in IX\), then every geodesic ray going out to $\zeta$ is irrational.\\
\item If \(\zeta\in IX\) and $\alpha$ and $\beta$ are geodesic rays going out to $\zeta$,
then the itineraries of $\alpha$ and $\beta$ eventually coincide.\\
\end{enumerate}
\end{corollary}

A geodesic space is said to have the \textit{geodesic extension property} if every geodesic segment
can be extended to a geodesic line.  As is true with the original Croke-Kleiner construction,
the blocks we construct will satisfy the geodesic extension property.\\

\begin{lemma}
\label{le:denseRX}
If blocks have the geodesic extension property, then $RX$ is dense.
\end{lemma}

\begin{proof}
Let $\alpha$ be an irrational geodesic ray and write \(\Itin\alpha=[B_1,B_2,...]\).
For each \(n\ge 1\), let $t_n$ be a time at which \(\alpha(t_n)\in B_n\).
Then every ray \(\alpha|_{[0,t_n]}\) can be extended to a geodesic ray \(\alpha_n\)
which does not leave the block \(B_n\).  Then \(\alpha_n\to\alpha\).
\end{proof}

Given a space $Y$, we call a map \(\phi:IX\to Y\) an \textit{irrational map} if it satisfies the
property that \(\phi(a)=\phi(b)\) iff whenever $\alpha$ and $\beta$ are geodesic rays
going out to \(a\) and \(b\) respectively, then \(\Itin\alpha\) and \(\Itin\beta\)
eventually coincide.
The obvious candidate for such a map is the function \(\phi:IX\to\partial\N\) which takes
$a$ to the boundary point in $\partial\N$ determined by the itinerary of a ray going out to $a$.
This function is well-defined by Corollary \ref{co:RXIX}(3).
All we need to know is that $\phi$ is continuous, which amounts
to proving the following lemma:\\

\begin{lemma}
Let \((\alpha_n)\) be a sequence of irrational rays with common basepoint converging to
another irrational ray $\alpha$.  Then for every \(B\in\Itin\alpha\), we have \(B\in\Itin\alpha_n\)
for large enough $n$.
\end{lemma}

\begin{proof}
Write \(\Itin\alpha=[B_1,B_2,...]\), and choose \(k\ge 1\).
Then \(B_{k+1}\) is a neighborhood of \(\alpha(t)\) for some time
$t$, which means that for large enough $n$, \(\alpha_n(t)\in B_{k+1}\).
Since \(\alpha_n|_{[0,t]}\) begins in $B_1$ and ends in $B_{k+1}$,
Lemma \ref{le:itins}(1) tells us that it must enter $B_k$.
\end{proof}

\begin{corollary}
\label{co:irrationalmap}
The natural map \(\phi:IX\to\partial\N\) determined by itineraries is an irrational map.
\end{corollary}

\section{Some Local Homology Calculations}
\label{sec:LHtechlemmas}
Another tool we will use will be singular homology, with \cite[p.108-130]{Ha} as our reference.
Here are a couple of key technical lemmas.  All homology will be computed using $\Z$ coefficients.\\

\begin{lemma}
\label{le:usepathcomponents}
Local homology can be computed using local path components.  That is, for
a point $x$ in a topological space $X$ with local path component $\pi$, we have
\[
	H_\ast(X,X-x)=H_\ast(\pi,\pi-x).
\]
\end{lemma}

\begin{proof}
Let $U$ be an open neighborhood of $x$ which has $\pi$ as a path component.
Using excision, we get
\[
	H_\ast(X,X-x)=H_\ast(U,U-x).
\]
Now, since the image of a singular simplex $\sigma$ is path connected,
the chain complex \(C_\ast(U)\) splits as \(C_\ast(U-\pi)\oplus C_\ast(\pi)\).
Passing to the relative chain complex \(C_\ast(U,U-x)\) kills off the entire first factor.
Thus \(C_\ast(U,U-x)=C_\ast(\pi,\pi-x)\) and
\(H_\ast(U,U-x)=H_\ast(\pi,\pi-x)\).
\end{proof}

\begin{lemma}
\label{le:homologyofaproductwithB}
Let $X$ be a path connected space, \(x_0\in X\), and $B$ be the open $n$-ball with \(z_0\in B\).
Then the local homology \(H_\ast(B\times X,B\times X-(z_0,x_0))\) is zero whenever \(\ast<n\).  When \(\ast=n\),
it is nonzero only when $X$ is a single point, in which case it is $\Z$.
\end{lemma}

\begin{proof}
If $X$ is just a single point, then \(B\times X=B\) and the lemma is a standard fact.  So, we will assume $X$ has more than
one point.  Fix \(k\le n\).
We will prove that any map $\sigma$ of a pseudo-$k$-manifold
$M$ into $B\times X$ such that \((z_0,x_0)\notin\sigma(\partial M)\) can be homotoped off of this point via a homotopy
$H_t$ such that \(H_t(\partial M)\) misses \((z_0,x_0)\) at every time $t$.
We will write the coordinates of $\sigma$ as \((\sigma_B,\sigma_X)\).\\

First of all, we smooth $\sigma_B$.  We do this with a $\delta$-homotopy where \(\delta<d(z_0,\sigma(\partial M))\)
(that is, a homotopy which doesn't move points more than a distance of $\delta$).
Then we homotope $\sigma_B$ rel \(\partial M\) so that it is transverse to $z_0$.
If \(k<n\), then \(\sigma_B^{-1}(z_0)\) is empty.  If $k=n$, then
\(\sigma_B^{-1}(z_0)\) is finite.\\

Next, we homotope $\sigma_X$ so that there is an open $k$-ball \(U\subset\intr M\) such that \(\sigma_X(U)\) misses $x_0$.
If \(\sigma_X^{-1}(x_0)\neq M\), we don't have to do anything, and since \(\sigma_X^{-1}(x_0)\) is compact, we
can get \(U\subset\intr M-\sigma^{-1}(x_0)\) without any moves at all.
Suppose \(\sigma_X(M)=x_0\); then get a
path \(\alpha:[0,1]\to X\) such that \(\alpha(0)=x_0\) and \(\alpha(1)\neq x_0\) and an open $k$-ball \(U\subset\intr M\).
Let
\[
	G_t=(G^1_t,G^2_t):\overline U\to\overline U\times[0,1]
\]
be a homotopy such that
\begin{align*}
	G_0 &= (\id|_{\overline U},0), \\
	G_t|_{\partial U}  &= (\id|_{\partial U},0) \textrm{ for all times }t,\\
\textrm{and }
	G_1(\overline U) &= \partial U\times[0,1]\cup U\times\{1\}.
\end{align*}
With this we define \(G'_t=(\sigma_B\circ G^1_t,\alpha\circ G^2_t)\); note that now the $X$-coordinate of
$G'_1(U)$ misses $x_0$.  We retake $\sigma$ to be the new map
\[
	\sigma|_{M-U}\cup G'_1
\]
which is homotopic to the old one rel $\partial M$.\\

Finally, let \(H_t:M\to M\) be a homotopy through homeomorphisms rel $\partial M$ from the identity to a
homeomorphism $h$ such that
\(h(\sigma_B^{-1}(z_0))\subset U\).  Then \((\sigma_BH^{-1}_t,\sigma_X)\) is a homotopy rel \(\partial M\)
from $\sigma$ to a new map
\begin{align*}
	\sigma'
&=
	(\sigma_B',\sigma_X') \\
&=
	(\sigma_Bh^{-1},\sigma_X)
\end{align*}
with the property that
\begin{align*}
	\sigma_X'\bigl((\sigma_B')^{-1}(z_0)\bigr)
&=
	\sigma_X\bigl(h\sigma_B^{-1}(z_0)\bigr) \\
&\subset
	\sigma_X\bigl(U\bigr),
\end{align*}
which misses \(x_0\).  So $\sigma'$ misses the point $(z_0,x_0)$.
\end{proof}

We denote \textit{reduced homology}, as defined in \cite[p.110]{Ha} by \(\widetilde H_\ast\).\\

\begin{lemma}
\label{le:homologyofsuspensions}
Let $X$ be a space and $\Sigma$ denote suspension.
Then \(\widetilde H_\ast(\Sigma X)=\widetilde H_{\ast-1}(X)\).
\end{lemma}

\begin{proof}
If we write
\[
	\Sigma X=CX\cup CX
\]
where \(CX\cap CX=X\), then the Mayer-Vietoris sequence gives us for \(n\ge 1\),
\[
	...
		\to
	\widetilde H_n(CX)\oplus \widetilde H_n(CX)
		\to
	\widetilde H_n(\Sigma X)
		\to
	\widetilde H_{n-1}(X)
		\to
	\widetilde H_{n-1}(CX)\oplus \widetilde H_{n-1}(CX)
		\to
	...
\]
Since \(CX\) is contractible, the first and last term shown here disappear and we are left with the statement
of the lemma.
\end{proof}

\section{Hyperblocks}
\label{sec:hyperblocks}

For positive integers $m$ and $n$ and an infinite CAT(0) group $\Gamma$, define
\[
	G_0
=
	\bigl(\Gamma\times\Z^m\bigr)\ast_{\Z^m}\bigl(\Z^m\times\Z^n\bigr)
=
	\Z^m\times\bigl(\Gamma\ast\Z^n).
\]
and choose a geometric action of \(\Z^m\times\Z^n\) on \(E_0=\R^m\times\R^n\) (by translations).
We will denote the convex hull of the \(\Z^m\)-orbit of the origin
by \(E^m_-\) and the convex hull of the \(\Z^n\)-orbit of the origin by
\(E^n_+\) (these are just isometric copies of \(\R^m\) and \(\R^n\)).  The angle $\theta$ between
geodesics in the two subspaces is called the \textit{skew}, which can be
any number \(0<\theta\le\pi/2\).  The quotient of $E_0$ by the group action is an $m+n$ torus $T_0$
with $m$- and $n$- tori \(T^m_-=E^m_-/\Z^m\) and \(T^n_+=E^n_+/\Z^n\);
we call these the \textit{left-} and \textit{right-hand} subtori of $T_0$.\\

For a CAT(0) group $\Gamma$, let $\overline K=\overline K(\Gamma)$ be a compact nonpositively curved $K(\Gamma,1)$,
and denote its universal cover by $K$.
Choose a point \(\overline x\in\overline K\) and glue \(T^m_-\times\overline K\) to $T_0$
via the isometry \(T^m_-\times\{\overline x\}\to T^m_-\).  The resulting space is nonpositively
curved (\cite[Prop II.11.6(2)]{BH}) and is a \(K(G_0,1)\);
we denote it by \(\overline Y=\overline Y(\Gamma,m,n)\)
(see Figure \ref{fig:2tori}).
Its universal cover $Y$ is a CAT(0) $G_0$-space, which we call a \textit{hyperblock}.
Path components of \(p^{-1}(T_0)\) are called \textit{walls};
these are isometric copies of $E_0$.
The names \textit{hyperblock} and \textit{wall} are given because these spaces will play the role
in this paper that ``blocks'' and ``walls'' play in \cite{CK}.\\

\begin{figure}[h]
\includegraphics[height=1.5in,width=.8\textwidth]{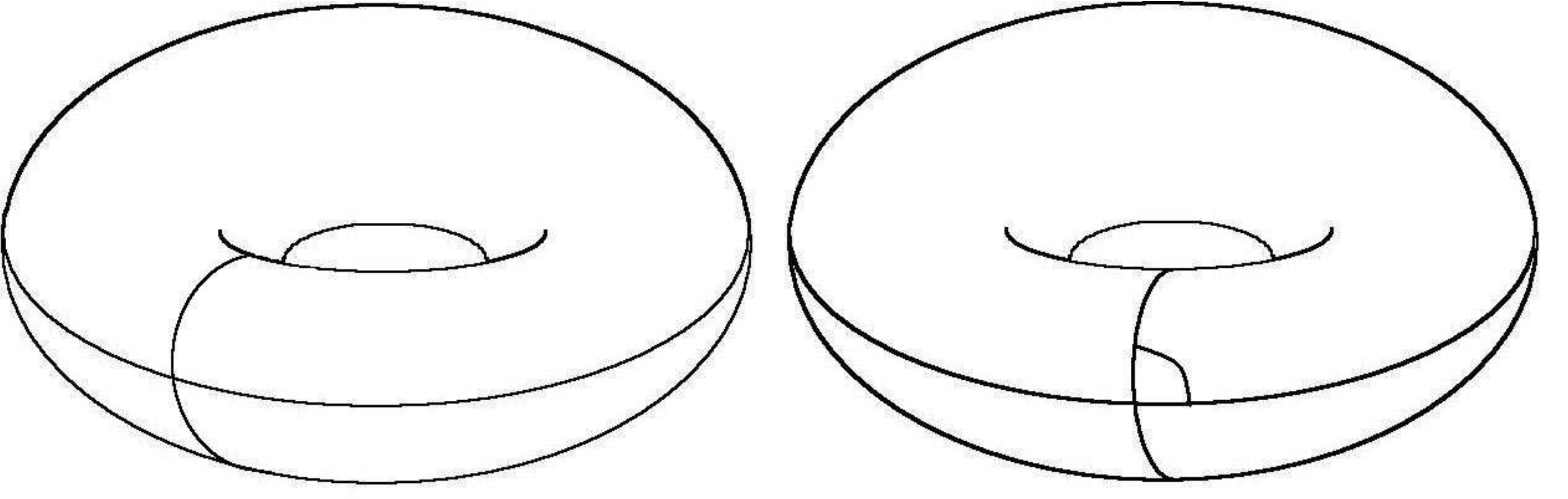}\\
\vspace{-1.2cm}\hspace{6.3cm}$\theta$\\
\vspace{-0.7cm}\hspace{-6.2cm}$T^m_-\times\{\overline x\}$\\
\vspace{-1.1cm}\hspace{-10.3cm}{\scriptsize $\overline K$}\\
\vspace{0cm}\hspace{4cm}$T^m_-$\\
\vspace{1cm}\(T^m_-\times\overline K\)
\hspace{5.5cm}\(T_0\)\\
\caption{$\overline Y$}
\label{fig:2tori}
\end{figure}

\subsection{The Hedge}
\begin{lemma}
\label{le:hedge}
\(Y\) splits as a product \(E^m_-\times H\) where \(H\) comes with a natural block structure in the sense
of Section \ref{sec:NervesandItineraries}.  In this block structure, left blocks are isometric copies of $K$
and right blocks are isometric copies of \(\R^n\).  The intersection of two blocks is at most a single point.
\end{lemma}

\begin{proof}
In the same way as in \cite{CK}, $Y$ comes with a natural block structure in which left blocks are path
components of \(p^{-1}(T^m_-\times\overline K)\) and right blocks are path components of \(p^{-1}(T_0)\).
We begin with a left block $B_0$ and an isometry
\[
	\phi_0:E^m_-\times K\to B_0
\]
which comes naturally from the splitting of the downstairs space.
Let \(\mathcal{B}_1\) denote the collection of right blocks
which intersect $B_0$.  For each \(B\in\mathcal{B}_1\), we have
\[
	B\cap B_0=\phi_0(E^m_-\times x_B)
\]
for some point \(x_B\in K\).
Now \((E^m_-)^\perp\) is the vector subspace of $E$ orthogonal to $E^m_-$.
There is also an isometry
\[
	\phi_1^B:E^m_-\times (E^m_-)^\perp\to B
\]
such that
\[
	B\cap B_0=\phi_1^B(E^m_-\times y_B)
\]
for some \(y_B\in (E^m_-)^\perp\).
We choose this isometry so that for every \(z\in E^m_-\),
\[
	\phi_1^B(z,y_B)=\phi_0(z,x_B).
\]
Define \(L_B=\phi_1^B((E^m_-)^\perp)\), \(y'_B=\phi_1^B(y_B)\),
\[
	D_1=B_0\cup\bigcup\mathcal{B}_1,
\]
and
\[
	H_1=K\cup\bigcup_{B\in\mathcal{B}_1}L_B/\sim
\]
where $\sim$ is generated by the rule \(x_B\sim y'_B\).
Then we can extend $\phi_0$ to an isometry
\[
	\phi_1:E^m_-\times H_1\to D_1
\]
by letting \(\phi_1(z,x)=\phi_1^B(z,x)\) whenever \(x\in L_B\).\\

Now, let \(\mathcal{B}_2\) be the collection of left blocks intersecting $D_1$.  As before, every
\(B\in\mathcal{B}_2\) intersects $D_1$ at a subspace of the form \(\phi_1(E^m_-\times x_B)\)
for some point \(x_B\in H_1\) and any natural isometry \(E^m_-\times K\to B\)
has the property that the intersection \(B\cap D_1\) is the image of \(E^m_-\times y_B\)
for some \(y_B\in K\).  Denote the image of \(0\times K\) under this map by \(L_B\) and the image of $(0,y_B)$ by
$y'_B$.  Let
\[
	D_2=D_1\cup\bigcup\mathcal{B}_2
\]
and
\[
	H_2=H_1\cup\bigcup_{B\in\mathcal{B}_2}L_B/\sim
\]
where $\sim$ is generated by the rule \(x_B\sim y'_B\) for \(B\in\mathcal{B}_2\).
In the same was above, $\phi_1$ can be extended to an isometry
\[
	\phi_2:E^m_-\times H_2\to D_2.
\]
Continue in this manner to get for every $n$ an isometry \(\phi_n:E^m_-\times H_n\to D_n\)
where $H_n$ has the desired block structure.
\end{proof}

Note that this splitting corresponds directly to the group splitting \(G_0=\Z^m\times[\Gamma\ast\Z^n]\) where \(\Z^m\)
acts only in the \(E^m_-\)-coordinate.  Then the projected action of \(\Gamma\ast\Z^n\) on $H$ is a geometric action
and the stabilizer of a block of $H$ is a subgroup conjugate to either $\Gamma$ or $\Z^n$.  A good example to keep in
mind at this point is \(G_0=\Z\times[\Z\ast\Z]\) (as in the Croke-Kleiner group).  In this case \(Y=\R\times H\) where
$H$ is the infinite 4-valent tree.\\

From the splitting \(Y=E^m_-\times H\), we get that \(\partial Y\) is the join
\(\partial E^m_-\ast\partial H\).  We call the points of \(\partial E^m_-\) \textit{poles} of
\(Y\).  The set of poles, which we denote $PY$, is an $(m-1)$-sphere.
We call \(H\) a \textit{hedge} and the blocks described in Lemma \ref{le:hedge} of the hedge \textit{leaves}.
Those leaves $L$ for which the space \(E^m_-\times L\) is a wall (a path component of
\(p^{-1}(T_0)\)) are called \textit{right leaves} (these are isometric copies of \(\R^n\)).
The other leaves are called \textit{left leaves} (these are isometric copies of $K$).
The points of intersction of leaves are called \textit{gluing points}.
Since the set of gluing points in a leaf is the orbit of a single one by a geometric action
(of either $\Z^n$ or $\Gamma$), this set is discrete and quasi-dense in the leaf.
With this block structure in hand, we can talk about itineraries of geodesics and geodesic rays
in the hedge in terms of leaves.  We will denote the itinerary of a geodesic $\alpha$ in $H$
by \(\Itin_H\alpha\).  The set of irrational points of \(\partial H\) will be denoted by \(IH\) and the
set of rational points by \(RH\).
Geodesics in the hedge are easy to compute:
If \(\Itin_H[x,y]=[L_1,...,L_n]\) and \(y_i=L_i\cap L_{i+1}\) for \(1\le i<n\), then
\[
	[x,y]=[x,y_1]\cup[y_1,y_2]\cup...\cup[y_{n-1},y]
\]
where each of these segments is taken in a leaf.\\

Given a CAT(0) space $X$ with points \(p,x,y\in X\), the \textit{Alexandrov angle between
the geodesics \([p,x]\) and \([p,y]\)} is defined to be the angle between the initial velocities
of these geodesics; this number is denoted by \(\angle_p(x,y)\).  If $\alpha$ and $\beta$ are
geodesic segments [rays] based at $p$, then we may denote the Alexandrov angle between them by
\(\angle_p(\alpha,\beta)\).  If $\alpha$ and $\beta$ are rays with endpoints
\(\zeta\) and \(\eta\in\partial X\), then
we may also write \(\angle_p(\zeta,\eta)=\angle_p(\alpha,\beta)\).
Given \(\zeta,\eta\in\partial X\), the
\textit{Tits} angle between $\zeta$ and $\eta$ is defined to be
\[
	\angle_{Tits}(\zeta,\eta)=\sup_{p\in X}\angle_p(\zeta,\eta)\textrm{ (see \cite[Defn 9.4]{BH})}.
\]
In \cite[Co 9.9]{BH}, it is shown that for geodesic rays $\alpha$ and $\beta$,
\[
	\angle_{Tits}(\alpha(\infty),\beta(\infty))=\angle_p(\alpha,\beta)
\]
iff $\alpha$ and $\beta$ bound a flat sector (or their union is a geodesic line, in which case this angle is \(\pi/2\)).
For closed subspaces $C$ and $D$ of \(\partial X\), we may also write
\[
	\angle_{Tits}(C,D)
=
	\min\bigl\{
		\angle_{Tits}(\zeta,\eta)
	\big|
		\zeta\in C,\eta\in D
	\bigr\}.
\]
\indexspace

In the context of this section, we will let
\[
	q:PY\times\partial H\times[0,\pi/2]\to\partial Y
\]
denote the natural quotient map where \(q(\eta,\zeta,0)=\eta\) and \(q(\eta,\zeta,\pi/2)=\zeta\).
That is, $\partial H$ disappears at level 0 and $PY$ disappears at level $\pi/2$, and $t$ is the Tits
angle between \(q(\eta,\zeta,0)\) and \(q(\eta,\zeta,t)\).\\

\begin{lemma}
\label{le:qcomputation}
Consider the induced actions of \(G_0\) on \(\partial Y\) and \(\Gamma\ast\Z^n\) on \(\partial H\).
For \(g\in G_0\), and \(q(\eta,\zeta,t)\in\partial Y\), we have
\[
	gq(\eta,\zeta,t)=q(\eta,g_2\zeta,t)
\]
where $g_2$ is the image of $g$ under the isomorphism \(G_0/\Z^m\to\Gamma\ast\Z^n\).
\end{lemma}

\begin{proof}
Given a point \(p=(p_1,p_2)\in Y=E^m_-\times H\) and points \(\eta\in PY\) and \(\zeta\in\partial H\), let
$\alpha$ and $\beta$ be geodesic rays based at $p$ going out to $\eta$ and $\zeta$ respectively.
Then we can write \(\alpha(t)=(\alpha_1(t),p_2)\) and \(\beta(t)=(p_1,\beta_2(t))\) where
$\alpha_1$ is a ray in $E^m_-$ and $\beta_2$ is a ray in $H$.  Because \(Y=\Min\Z^m\), every $g\in G_0$
can be written coordinate-wise as \((g^E,g^H)\) where \(g^E\)
is a translation of $E^m_-$ and $g^H$ is an isometry of $H$.  Of course, since $g_1$ acts only in the $E^m_-$-coordinate
of $Y$, \(g^H=g_2^H\).  Using this, we compute
\begin{align*}
	d\bigl(g\alpha(t),\alpha(t)\bigr)
&=
	\sqrt
	{
		d_{E^m_-}\bigl(g^E\alpha_1(t),\alpha_1(t)\bigr)^2
			+
		d_{H}(g^Hp_2,p_2)^2
	} \\
&=
	\sqrt
	{
		d_{E^m_-}\bigl(g^Ep_1,p_1\bigr)^2
			+
		d_{H}\bigl(g^Hp_2,p_2\bigr)^2
	} \\
&=
	d(gp,p),
\end{align*}
which means that \(g\eta=\eta\), and
\begin{align*}
	d\bigl(g\beta(t),g_2\beta(t)\bigr)
&=
	\sqrt
	{
		d_{E^m_-}(g^Ep_1,p_1)^2
			+
		d_{H}\bigl(g^H\beta_2(t),g_2^H\beta_2(t)\bigr)^2
	} \\
&=
	d_{E^m_-}(g^Ep_1,p_1),
\end{align*}
which means that \(g\zeta=g_2\zeta\).
Since $g$ takes the flat quadrant
%((
\(\alpha[0,\infty)\times\beta[0,\infty)\) %]]
to the flat quadrant
%((
\(g\alpha[0,\infty)\times g\beta[0,\infty)\), %]]
the lemma follows.
\end{proof}

When \(g\in\Gamma\ast\Z^n\), we will simply write the equation in this lemma as
\[
	gq(\eta,\zeta,t)=q(\eta,g\zeta,t),
\]
since there is no confusion.\\

\begin{lemma}
$IH$ is dense in \(\partial H\).
\end{lemma}

\begin{proof}
We begin by proving the following claim.\\

\begin{claim}
For distinct points \(x,y_1\in H\) such that $y_1$ is a gluing point, the geodesic \([x,y_1]\)
can be extended to an irrational ray.
\end{claim}

Let $L_1$ be the leaf which intersects the geodesic \([x,y_1]\) only at the point $y_1$.
Since the set of gluing points in $L_1$ is quasi-dense (and $\Gamma$ is infinite),
we can choose another gluing point \(y_2\in L_1-y_1\),
say \(y_2=L_1\cap L_2\).  Then \([x,y_2]=[x,y_1]\cup[y_1,y_2]\).
Get another gluing point \(y_3\in L_2-y_2\), say \(y_3=L_2\cap L_3\), so that
\[
	[x,y_3]=[x,y_1]\cup[y_1,y_2]\cup[y_2,y_3].
\]
Continue in this manner, always extending the geodesic into a new leaf to get an
irrational ray.  This proves the claim.\\

Now, let \(\alpha\) be any rational ray in $H$ based at a point \(x\), where $L$ is the
last leaf \(\alpha\) enters.
Get a sequence \(z_n\) of gluing points in $L$ converging to \(\alpha(\infty)\) in
\(L\cup\partial L\).  By the claim, we can extend every geodesic \([x,z_n]\) to an
irrational ray \(\alpha_n\).  Then \(\alpha_n\to\alpha\).
\end{proof}

It turns out that $RH$ is also dense in \(\partial H\).  In fact, something much stronger is
true:\\

\begin{lemma}
\label{le:density}
Let \(\mathcal{L}\) be either the collection of left leaves or right leaves
and \(\mathcal{S}=\{\zeta_L\}_{L\in\mathcal{L}}\) be a subset of \(\partial H\)
obtained by choosing a single point \(\zeta_L\) from every leaf \(L\in\mathcal{L}\).
Then \(\mathcal{S}\) is dense in \(\partial H\).
\end{lemma}

\begin{proof}
We show that every irrational point is a limit point of \(\mathcal{S}\).  Choose any
irrational ray \(\alpha\) based at a non-gluing point \(x\in L_1\) with infinite itinerary \([L_1,L_2,...]\).
For each $i$, denote the gluing point \(L_i\cap L_{i+1}\) by \(y_i\).  Then either all the even
leaves or all the odd leaves are from \(\mathcal{L}\); we'll assume it's the evens.
Let \(\alpha_i\) be the geodesic ray based at $x$ which goes out to \(\zeta_{L_{2i}}\).
Then since each $\alpha_i$ eventually stays in \(L_{2i}\), it agrees with \(\alpha\) along the segment
\([x,y_{2i-1}]\), and \(\alpha_i\to\alpha\).
\end{proof}

\begin{lemma}
\label{le:hedgebdry}
If two geodesic rays in a hedge have different itineraries, then they lie in different
path components of \(\partial H\).
\end{lemma}

\begin{proof}
It suffices to prove the following claim.\\

\begin{claim}
Choose a basepoint $p\in H$ (not a gluing point), and let $L$ be a leaf not
containing $p$.  Let \(\Omega\) be the collection of geodesic rays that have $L$ in their itineraries,
and \(\Omega(\infty)\subset\partial H\) be the set of endpoints of rays from \(\Omega\).  Then \(\Omega(\infty)\)
is both open and closed.
\end{claim}

Let $y$ be the gluing point at which every ray in $\Omega$ enters $L$, and get \(\epsilon>0\)
so that the open $\epsilon$-ball $B_\epsilon(y)$ in $L$ based at
$y$ contains no other gluing points.
Then \(\Omega(\infty)\) is open because $\Omega$ is the collection of rays passing through the open space
\(B_\epsilon(y)\cap L-y\).  It is closed, since whenever \((\alpha_n)\) is a sequence of rays in \(\Omega\)
converging to a ray \(\alpha\), the sequence of points where \(\alpha_n\) intersects the closed
space \(\partial B_\epsilon(y)\cap L\) converges to a point on \(\partial B_\epsilon(y)\cap L\),
showing that $\alpha$ also enters $L$.
\end{proof}

\begin{corollary}
\hspace{0cm}\\
\begin{enumerate}
\item Every irrational point is a path component of \(\partial H\).\\
\item \(\partial H\) is nowhere locally path connected.\\
\item \(\partial Y\) is locally path connected precisely at its poles.\\
\end{enumerate}
\end{corollary}

For \(\zeta\in\partial H\), we define the \textit{longitude} of \(Y\) at \(\zeta\) to be the subspace
\[
	l(\zeta):= PY\ast\zeta
\]
of $\partial Y$.
Since \(PY\) is an \((m-1)\)-sphere, $l(\zeta)$ is an \(m\)-ball.
\(\Gamma\ast\Z^n\) permutes right leaves transitively and for every right leaf $L$ of $H$, the wall \(W_L=E^m_-\times L\) is
stabilized by the conjugate subgroup
\[
	G_L=[\Z^m\times\Z^n]_{g_L}=g_L[\Z^m\times\Z^n]g_L^{-1}
\]
for some element $g_L\in\Gamma\ast\Z^n$.
The limit set of \([\Z^n]_{g_L}\) (that is, the set of limit points
in \(\partial W_L\) of the orbit of a single point in $W_L$)
will be denoted by \(S_L\); this is an $(n-1)$-sphere.  Note that since the action of
\(G_L\) on \(W_L\) comes from the original action of \(\Z^n\times\Z^m\)
on $E$, we have
\[
	\angle_{Tits}(PY,S_L)=\theta,
\]
where $\theta$ is the skew.
Consider the subset
\[
	\Pi Y=\bigcup_{\textrm{right leaves }L}S_L.
\]
Later we will see that \(\Pi Y\) is the collection of poles of hyperblocks which intersect $Y$ at a wall.
The following rule is easy to verify:
\[
	\lim[\Z^n]_{g_L}=g_L\lim\Z^n.
\]
Using this and the fact that \(\Gamma\ast\Z^n\) permutes right leaves transitively,
we get\\

\begin{fact}
\label{fact:spherepermutation}
\(\Gamma\ast\Z^n\) permutes the spheres of $\Pi Y$ transitively.
\end{fact}

The topological object which we will use to distinguish between boundaries is called the
\textit{watermark} \(\mathcal{W}(Y)\).
It is defined as follows: let \(L_0\) be the right leaf which is stabilized
by \(\Z^n\).  Then \(\mathcal{W}(Y)\) is the image of $S_{L_0}$ under the map
\(\partial W_{L_0}\to l(\zeta_0)\) which is given by the rule
\[
	q(\eta,\zeta,t)\mapsto q(\eta,\zeta_0,t)
\]
for some (any) fixed \(\zeta_0\in\partial H\).  The following lemma guarantees that this
is well-defined.\\

\begin{lemma}
For every $\zeta_0\in\partial H$, \(l(\zeta_0)\cap\overline{\Pi Y}\approx\mathcal{W}(Y)\).
\end{lemma}

\begin{proof}
We begin with the forward inclusion;
assume we have been given \(\nu\in l(\zeta_0)\cap\overline{\Pi Y}\), say
\[
	\nu=q(\eta,\zeta_0,\lambda).
\]
Then there is a sequence \((\nu_n)\subset\Pi Y\) converging to \(\nu\),
say \(\nu_n=q(\eta_n,\zeta_n,\lambda_n)\).  For each $n$, let \(S_{L_n}\) be the sphere
of \(\Pi Y\) containing $\nu_n$.  By Lemma \ref{fact:spherepermutation},
we can get a sequence \((g_n)\subset\Gamma\ast\Z^n\) for which \(g_n\nu_n\in S_{L_0}\) for all $n$.
by passing to a subsequence, we may assume that \(g_n\nu_n\to\nu'\in S_{L_0}\).
Write \(\nu'=q(\eta',\zeta',\lambda')\).
Using Lemma \ref{le:qcomputation}, we have \(g_n\nu_n=q(\eta_n,g_n\zeta_n,\lambda_n)\).
If \(\lambda<\pi/2\), then we must have \(\eta_n\to\eta\) and \(\eta'=\eta\).
If \(\lambda=\pi/2\), then we have \(q(\eta',\zeta',\lambda)=q(\eta,\zeta',\lambda)\).
Either way,
\begin{align*}
	\nu'
&=
	\lim_{n\to\infty}g_n\nu_n \\
&=
	q(\eta,\zeta',\lambda),
\end{align*}
showing that \(q(\eta,\zeta',\lambda)\in S_{L_0}\) and \(\nu\in\psi(S_{L_0})\).\\

For the reverse inclusion, assume we have \(\nu=q(\eta,\zeta',\lambda)\in S_{L_0}\).
By Lemma \ref{le:density}, the \(\Gamma\ast\Z^n\)-orbit of \(\zeta'\)
is dense in \(\partial H\), which means that we can get a sequence
\((g_n)\subset\Gamma\ast\Z^n\) such that \(g_n\zeta'\to\zeta_0\).
This done, we have
\begin{align*}
	\lim_{n\to\infty}g_n\nu
&=
	\lim_{n\to\infty}q(\eta,g_n\zeta',\lambda) \\
&=
	q(\eta,\zeta_0,\lambda),
\end{align*}
showing that \(q(\eta,\zeta_0,\lambda)\in\overline{\Pi Y}\), as desired.
\end{proof}

In the original Croke-Kleiner construction,
the watermark of every block is exactly two points if the skew is less than $\pi/2$ and one point if the skew is
precisely $\pi/2$.  The appropriate analogue in this setting is the following.\\

\begin{prop}
\label{prop:watermark}
The watermark of $Y$ contains exactly one point iff the skew is $\pi/2$.
\end{prop}

\subsection{Local Homology of $\partial Y$}
\label{sec:localhomology}

Given a topological space $A$, we let $CA$ denote the cone on $A$.
We will also denote by \(C^oA\) the ``open cone'' on $A$; that is,
\[
	C^oA=CA-A.
\]

\begin{lemma}
Let \(\zeta\) be a pole of $Y$.  Then \(H_k(\partial Y,\partial Y-\zeta)\) is zero
when \(k<m\) and uncountably generated when \(k=m\).
\end{lemma}

\begin{proof}
Get an open neighborhood \(V\subset PY\) of $\zeta$ which is homeomorphic to the open \((m-1)\)-ball
and consider
%(
\[
	U=q\bigl(V\times\partial H\times[0,\pi/2)\bigr).
\] %]
This is an open neighborhood of $\zeta$ homeomorphic to
\[
	(B\ast\partial H)-\partial H
\approx
	C^o\partial H\times B
\]
where $B$ is the open $(m-1)$-ball.
Assume \(k>1\) and look at the long exact sequence for the pair \((U,U-\zeta)\).
Since $U$ is contractible, \(H_k(U)=H_{k-1}(U)=0\) and we get
\[
	H_k(U,U-\zeta)
=
	H_{k-1}(U-\zeta).
\]    %(
Now,
\[
	U-\zeta
\HE
	C\partial H\times\overline B-(p,x)
\]
where $p$ is the cone point of \(C\partial H\) and \(x\in B\).
This last deformation retracts onto the subspace
\begin{align*}
	(C\partial H\times\partial B)\cup(\partial H\times B)
&\approx
	\partial H\ast S^{m-2} \\
&=
	\Sigma^{m-1}\partial H
\end{align*}
(see Figure \ref{fig:CYtimesDHE}).
\begin{figure}
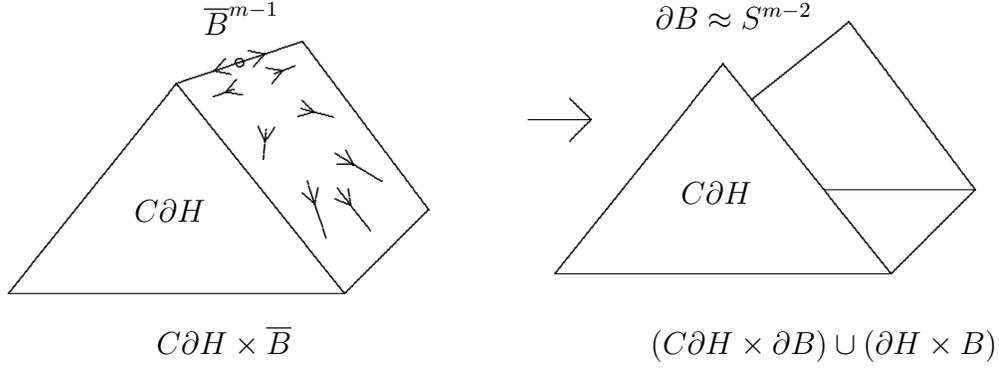

\input{CYtimesD.tex}\hspace{0cm}\\
\vspace{-3.5cm}\hspace{-0.5cm}
\input{rightarrow.tex}\hspace{0cm}\\
\vspace{-2.1cm}\hspace{5cm}
\input{CYtimesD2.tex}\hspace{0cm}\\
$C\partial H\times\overline B$
\hspace{4.5cm}
\((C\partial H\times\partial B)\cup(\partial H\times B)\)\\
\caption{A Deformation Retraction}
\label{fig:CYtimesDHE}
\end{figure}
So, applying Lemma \ref{le:homologyofsuspensions}, we get
\begin{align*}
	H_k(U,U-\zeta)
&=
	H_{k-1}(\Sigma^{m-1}\partial H) \\
&=
	\widetilde H_{k-1}(\Sigma^{m-1}\partial H) \\
&=
	\widetilde H_0(\Sigma^{m-k}\partial H) \\
&=
	0.
\end{align*}
Now consider what happens to the long exact sequence for the pair \((U,U-\zeta)\) when \(k=0,1\):\\
\[
	\to
		H_1(U)
	\to
		H_1(U,U-\zeta)
	\to
		H_0(U-\zeta)
	\to
		H_0(U)
	\to
		H_0(U,U-\zeta)
	\to
		0
\]
Because $U$ is path connected, we have \(H_1(U)=H_0(U,U-\zeta)=0\) and \(H_0(U)=\Z\).  So we are left with the short exact sequence
\[
	0
		\to 
	H_1(U,U-\zeta)
		\to 
	H_0(\Sigma^{m-1}\partial H)
		\to 
	\Z
		\to 
	0.
\]
If \(m>1\), then \(H_0(\Sigma^{m-1}\partial H)=\Z\) and \(H_1(U,U-\zeta)=0\).
If \(m=1\), then \(H_0(\Sigma^{m-1}\partial H)=H_0(\partial H)\) is uncountably generated and \(H_1(U,U-\zeta)\)
is also uncountably generated.
\end{proof}

\begin{lemma}
Let \(\zeta\in\partial Y-PY\).  Then \(H_k(\partial Y,\partial Y-\zeta)\) is
finitely generated for \(k\le m\).
\end{lemma}

\begin{proof}
We begin by finding a local path component of $\zeta$ which is homeomorphic to
\(B\times V_\zeta\) where $B$ is the open $m$-ball and \(V_\zeta\) is a local path
component in \(\partial H\).  If \(\zeta\notin\partial H\), then
this is easy, since $q$ is a homeomorphism on \(PY\times\partial H\times(0,\pi/2)\).
If \(\zeta\in\partial H\), then we can get an open neighborhood of $\zeta$ of the form
\[ %[
	U=q\bigl(PY\times V\times(0,\pi/2]\bigr)
\] %)
where $V$ is an open neighborhood of $\zeta$ in \(\partial H\).
Then
\[
	U
\approx
	C^oPY\times V
\approx
	B\times V,
\]
and the path component of $\zeta$ in $U$ has the form \(B\times V_\zeta\)
where $V_\zeta$ is the path component of $\zeta$ in $V$.
Therefore
\begin{align*}
	H_k(\partial Y,\partial Y-\zeta)
&=
	H_k(B\times V_\zeta,B\times V_\zeta-\zeta).
\end{align*}
The conclusion now follows from Lemma \ref{le:homologyofaproductwithB}.
\end{proof}

\begin{corollary}
\label{co:localhomology}
The local homology \(H_k(\partial Y,\partial Y-\zeta)\) at a point \(\zeta\in\partial Y\)
is finitely generated for \(k<m\).  For \(k=m\), it's infinitely generated iff $\zeta$ is a pole.
\end{corollary}

\section{The Generalized Croke-Kleiner Construction}
	\label{sec:generalCK}
Let $\Gamma_-$ and $\Gamma_+$ be infinite CAT(0) groups and \(m\le n\)
be positive integers, and define
\[
	G
=
	\bigl(\Gamma_-\times\Z^m\bigr)
		\ast_{\Z^m}
	\bigl(\Z^m\times\Z^n\bigr)
		\ast_{\Z^n}
	\bigl(\Z^n\times\Gamma_+\bigr).
\]
Choose \(T_0(m,n)=E/\Z^m\times\Z^n\) with skew $\theta$ and form the spaces
\(\overline K_- = \overline K(\Gamma_-)\), \(\overline Y_- = \overline Y(\Gamma_-,m,n)\).
We also form \(\overline Y_+ = \overline Y(\Gamma_+,n,m)\) where we use the same torus $T_0$
but with ``left'' and ``right'' swapped (this corresponds to a change of coordinattes in $K$)
so that \(\overline Y_-\cap\overline Y_+=T_0\).
We define \(\overline X=\overline Y_-\cup\overline Y_+\).\\

\begin{figure}[h]
\includegraphics[height=1.5in,width=1\textwidth]{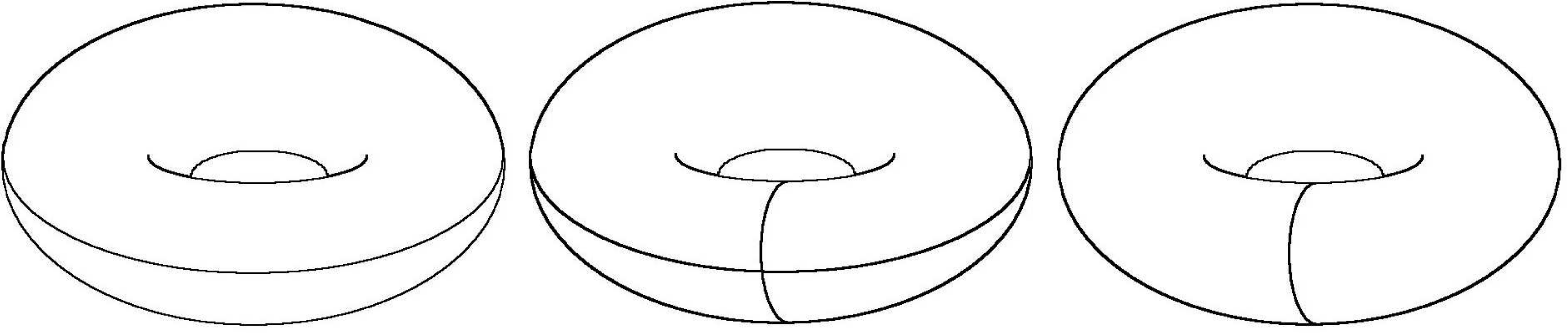}\\
\vspace{-1.4cm}\hspace{-11cm}$T^m_-$\\
\vspace{-0.8cm}\hspace{-1.3cm}$T^m_-$\hspace{1.5cm}$T^n_+$\\
\vspace{-0.1cm}\hspace{11.5cm}$T^n_+$\\
\vspace{1cm}\hspace{0cm}\(\overline K_-\times T^m_-\)
\hspace{4.5cm}\(T_0\)
\hspace{4.5cm}\(\overline K_+\times T^n_+\)\\
\caption{$\overline X$}
\label{fig:3tori}
\end{figure}

The universal cover $X$ of \(\overline X\) is a
CAT(0) $G$-space.  We denote the universal covering projection by $p$ and call $X$ a
\textit{generalized Croke-Kleiner construction for $G$ from the spaces $T_0$, $\overline{K_-}$, and $\overline{K_+}$.}
The path components of \(p^{-1}(T_0)\) are isometric copies of $E$; we call these \textit{walls}.  The path
components of \(p^{-1}(Y_-)\) and \(p^{-1}(Y_+)\) are hyperblocks; we call them \textit{left} and \textit{right blocks} respectively,
and denote them by $B_-$ and $B_+$ if we wish to designate parity.
The following facts are easy.\\

\begin{fact}
If a wall intersects a block, then that wall is contained in that block.\\
\end{fact}

\begin{fact}
If two blocks intersect, then they have opposite parity and their intersection is precisely a wall.\\
\end{fact}

\begin{fact}
Every block $B$ splits as a product \(E\times H\) where \(E\) is a Euclidean space and \(H\) is a hedge.
\end{fact}

In this last fact, $E$ and $H$ are of two types depending on the parity of $B$.  If $B$ is a left block,
then $E$ is an isometric copy of \(E^m_-\) and $H$ is a hedge whose leaves are copies of $K_-$
and $E^m_-$.  If $B$ is a right block, then \(E\) is a copy of \(E^n_+\) and
the leaves of $H$ are copies of $K_+$ and $E^n_+$.  We call
$H$ the \textit{hedge factor} of $B$.\\

\subsection{The Block Structure}
To see that the collection of hyperblocks satisfies Definition \ref{defn:blockstructure}, we need only the
following.\\

\begin{lemma}
There is an \(\epsilon>0\) such that two blocks intersect iff their $\epsilon$-neighborhoods intersect.
\end{lemma}

\begin{proof}
Take \(\epsilon>0\) to be smaller than half
the lengths of the shortest nontrivial loops in \(\overline{K}_-\) and \(\overline{K}_+\).
Suppose $B_1$ and $B_2$ are disjoint blocks, and
let \(\gamma\) be a geodesic starting in $B_1$ and ending in $B_2$.  Without loss of generality,
assume \(B_1\) is a left block.
Then \(\gamma'=p\gamma\) is a local geodesic in \(\overline{X}\) which leaves \(\overline{Y}_-\) at a point  %[
\(\gamma'(t_0)\in T^n_+\).  If \(\gamma'|_{(t_0,1]}\) stays in \(\overline K_+\times T^n_+-T^n_+\),  %)
then $B_2$ intersects $B_1$ at \(\gamma(t_0)\).  So \(\gamma'|_{[t_0,1]}\) reenters \(\overline{Y}_-\), at another
point \(\gamma'(t_1)\in T^n_+\).  Then \(\gamma'|_{(t_0,t_1)}\subset\overline K_+\times T^n_+\)
and its projection onto the \(\overline{K_+}\)-coordinate is a nontrivial loop in \(\overline{K}_+\).
Since projections do not increase distance, it follows that the length of $\gamma$ is also
at least \(2\epsilon\), guaranteeing that \(d(B_1,B_2)\ge\epsilon\).
\end{proof}

As a corollary, we get\\

\begin{ThmA'}
The nerve $\N$ of blocks is a tree.
\end{ThmA'}

\begin{fact}
Given a block $B$, \(\Pi B\), as defined in the previous section, is precisely the set of poles of blocks neighboring $B$.
\end{fact}

\begin{ThmB'}
Let \(B_0\) and \(B_1\) be blocks, and $D$ be the distance between the corresponding vertices in
$N$.  Then:\\
\begin{enumerate}
\item If \(D=1\), then \(\partial B_0\cap\partial B_1=\partial W\) where $W$ is the wall \(B_0\cap B_1\).\\
\item If \(D=2\), then \(\partial B_0\cap\partial B_1=PB_{1/2}\) where
\(B_{1/2}\) intersects \(B_0\) and \(B_1\).\\
\item If \(D>2\), then \(\partial B_0\cap\partial B_1=\emptyset\).\\
\end{enumerate}
\end{ThmB'}

\begin{proof}
(1) If \(D=1\), then \(B_0\cap B_1\) is a wall $W$.  That \(\partial W\subset\partial B_0\cap\partial B_1\)
is obvious.  To see the reverse inclusion, suppose we have asymptotic geodesic rays
\(\alpha_0\subset B_0\) and \(\alpha_1\subset B_1\).  Every
geodesic from $\alpha_0$ to $\alpha_1$ intersects the wall $W$ (Lemma \ref{le:itins}).
Thus we can get a sequence of points in $W$ which remain asymptotic to $\alpha_0$ and $\alpha_1$.\\

(2) If \(D=2\), then there is one vertex between \(v_{B_0}\) and \(v_{B_1}\); call it
\(v_{B_{1/2}}\).
We will show that
\[
	PB_{1/2}
		\subset
	\partial B_0\cap\partial B_1
		\subset
	\partial W_0\cap\partial W_1
		\subset
	PB_{1/2}
\]
where \(W_i=B_{1/2}\cap B_i\) for \(i=0,1\).
The first inclusion follows naturally from
the fact that geodesic rays in \(B_{1/2}\) which go to poles are precisely those whose
projections onto the hedge coordinate of \(B_{1/2}\) are constant, and can be therefore be
constructed easily in \(W_0\) and \(W_1\).
The second inclusion follows by the same argument as in (1).
For the third inclusion, suppose \(\alpha_0\subset W_0\) and \(\alpha_1\subset W_1\) are asymptotic geodesic rays,
and let \(\overline\alpha_0\) and \(\overline\alpha_1\) be their projections
onto the hedge coordinate of \(B_{1/2}\).  Let \(L_0\) and \(L_1\) be the leaves containing the images of
these two maps.  Since $L_0$ and $L_1$ are disjoint, every geodesic from $\overline\alpha_0$
to \(\overline\alpha_1\) leaves $L_0$ at the same gluing point and enters $L_1$ at the same gluing point.
Hence, the only way for $\alpha_0$ and $\alpha_1$ to be asymptotic is if
$\overline\alpha_0$ and $\overline\alpha_1$ are both constant.
Therefore \(\alpha_0\) and \(\alpha_1\) go to a pole of \(B_{1/2}\).\\

Finally, we show (3) by contradiction: Suppose \(\zeta\in\partial B_0\cap\partial B_1\),
and write \(\Itin[B_0,B_1]=[\overline B_1,...,\overline B_n]\) where \(n=D+1\) by
hypothesis.
By the same argument as in (1), we actually have that
\(\zeta\in\partial B_i\) for every \(1\le i\le n\).  By (2), then, it follows that
\(\zeta\in PB_i\) for every \(1<i<n\).  But \(PB_2\cap PB_3=\emptyset\), because \(\angle_{Tits}(PB_2,PB_3)=\theta\),
giving us a contradiction!
\end{proof}

\subsection{Poles and $n$-Vertices}
In \cite{CK}, a boundary point $\zeta$ is called a \textit{vertex} if it has a local path component $\pi$ and a homeomorphism
from $\pi$ to the open cone on the cantor set taking $\zeta$ to the cone point.  An appropriate analogue in this
context is the following: A point \(\zeta\in\partial X\) is called a \textit{vertex} if the local homology of \(\partial X\)
at $\zeta$ is uncountably generated in some dimension.  If $n$ is the smallest
dimension in which this local homology is uncountably generated, then we say that $\zeta$ is an \textit{$n$-vertex}.
The goal of this section is to distinguish topologically which vertices in $RX$ are poles.
A key tool is the following:\\

\begin{ThmC'}
Let $B$ be a block and \(\zeta\in\partial B\) not be a pole of any neighboring block.  Then $\zeta$ has a local
path component which stays in \(\partial B\).
\end{ThmC'}

The proof of this theorem is the same as the proof of Theorem C (\cite[Le 4]{CK}) with the following observation.
First of all, the topological frontier of a left block is
a subcollection of path components of \(p^{-1}(T^n_+)\), and the topological frontier
of a right block is a subcollection of path components of \(p^{-1}(T^m_-)\); these path
components are isometric copies of \(E^n_+\) and \(E^m_-\) and are the appropriate
replacements for the ``singular geodesics'' given the original proof.\\

Recall that $m$ and $n$ are the dimensions of \(E^m_-\) and \(E^n_+\) respectively and that \(1\le m\le n\).\\

\begin{lemma}
\hspace{0cm}\\
\label{le:fewfakepoles}
\begin{enumerate}
\item If $m=n$, then $m$-vertices in $RX$ are poles.\\
\item If \(m<n\), then $m$-vertices in $RX$ are poles of left blocks.\\
\item If \(m<n\), then $n$-vertices in right block boundaries are poles of right blocks.\\
\end{enumerate}
\end{lemma}

\begin{proof}
Choose any \(\zeta\in RX\).
Recall that in Section \ref{sec:NervesandItineraries}, we showed that the collection of rational points
$RX$ of \(\partial X\) is the same as the union of block boundaries.
So there is a block $B$ such that \(\zeta\in\partial B-\Pi B\).
If $B$ is a left block, then let \(k=m\), and if $B$ is a right block, let \(k=n\).
Applying Theorem C$'$, Lemma \ref{le:usepathcomponents}, and Corollary \ref{co:localhomology}, we know
that $\zeta$ is a $k$-vertex iff $\zeta$ is a pole (of $B$).
\end{proof}

Now, it is concievable in the case where $m<n$ that some $n$-vertices of left block boundaries are not poles.
Here is the last resort for dealing with this situation.\\

\begin{lemma}
\label{le:fakerightpoles}
Assume \(m<n\) and let $\pi$ be a path component in the set of $n$-vertices in $RX$.
Then:\\
\begin{enumerate}
\item If $\pi$ is compact, then every point of $\pi$ is a pole.\\
\item If $\pi$ is not compact, then no point of $\pi$ is a pole.\\
\end{enumerate}
\end{lemma}

\begin{proof}
We begin by proving that if one point of $\pi$ is a pole, then every other point of $\pi$ is a pole
as well:
Suppose \(\alpha:[0,1]\to\pi\) is a path such that \(\alpha(0)\) is a pole
and \(\alpha(1)\) is not.  Since the collection of poles is closed, we may write
\[
	\overline t=\max\bigl\{0\le t\le 1\big|\alpha(t)\textrm{ is a pole}\bigr\};
\]
say \(\alpha(\overline t)\in PB_+\).
By Theorem C$'$, there is an \(s\in(\overline t,1)\) so that \(\alpha(\overline t,s)\subset\partial B_+\)
and contains no $n$-vertices, giving us a contradiction.\\

We now prove (1) and (2) by contrapositive; if $\pi$ contains all poles, then $\pi$ is a sphere,
giving us (2).  For (1), assume $\pi$ contains no poles; then
\begin{align*}
	\pi
&\subset
	\partial B-PB \\
&=
	PB\ast\partial H(B)-PB \\
&\approx
	\partial H\times B^m
\end{align*}
where $B$ is a left block, $B^m$ is the open $m$-ball, and $H$ is the hedge factor of $B$.
If \((\zeta,z)\in\pi\), then since every other point
of \(\{\zeta\}\times B^m\) has a local path component homeomorphic to that of
\((\zeta,z)\), we know that in fact \(\{\zeta\}\times B^m\subset\pi\).
But the frontier of \(\{\zeta\}\times B^m\) in \(\partial B\) is \(PB\),
which is disjoint from $\pi$.  This proves (1).
\end{proof}

\subsection{Safe Path Components}
\label{sec:spcomps}
In their original work, Croke and Kleiner conceived that \(RX\) was not a path component.
However, in order for a path starting in $RX$ to leave it, it has to pass through infinitely
many block boundaries on its way out.  The same is true here.  In this section, we define a
\textit{safe path} to be a path $\alpha(t)$ in \(\partial X\) which passes through $m$-vertices
for only finitely many times $t$.
Recall that under our assumption that \(m\le n\), it follows from Corollary \ref{co:localhomology}
that $m$-vertices are poles.  If $m<n$, then these are poles of left blocks.\\

\begin{ThmD'}
$RX$ is the unique dense safe path component of \(\partial X\).
\end{ThmD'}

\begin{proof}
We begin by showing that $RX$ is safe path connected; choose \(\zeta,\eta\in RX\),
say \(\zeta\in\partial B_0\) and \(\eta\in\partial B'_0\).  We will construct a safe path between
$\zeta$ and $\eta$ in $RX$ by induction on the length of \(\Itin[B_0,B'_0]=[B_1,...,B_n]\).
Assume \(n=1\) (that is, that \(B_0=B_0'\)), and let \([\zeta_P,\zeta_H]\) and \([\eta_P,\eta_H]\) be join arcs of
\(\partial B_0=PB_0\ast\partial H_0\) containing $\zeta$ and $\eta$.
Note that each of these join arcs contains at most two poles, one of which is a pole of $B_0$.
For, if one contained two poles of a nieghboring block, then we would not have \(E^m_-\cap E^n_+=\{0\}\).
Take $\alpha$ to be the path from $\zeta$ to $\eta$
\[
	[\zeta,\zeta_P]\cup[\zeta_P,\eta_H]\cup[\eta_H,\eta].
\]
Then $\alpha$ passes through only finitely many left poles, and hence is safe.\\
In general, if \(n>1\), then choose a point \(\zeta'\in\partial B_{n-1}\cap\partial B_n\) and
concatenate a safe path from $\zeta$ to $\zeta'$ in \(\partial B_1\cup ...\cup\partial B_{n-1}\)
to a safe path from $\zeta'$ to $\eta$ in \(\partial B_n\).\\

So $RX$ is safe path connected.  The next step is to show that $RX$ is a safe path component.  We do this by
contradiction:
Suppose we have a safe path $\alpha$ which starts in
$RX$ but ends outside.  Set \(\overline t=\inf\{t|\alpha(t)\notin RX\}\).  If \(\alpha(\overline t)\in RX\),
then \(\alpha(\overline t)\in\partial B\) for some block $B$ and
\(\alpha(\overline t,\overline t+\epsilon)\subset\partial B\) for some small \(\epsilon>0\).  So \(\alpha(\overline t)\notin RX\), but %(
\(\alpha[0,\overline t)\subset RX\). %]
Let $s$ be the last time at which $\alpha$ passes through an $m$-vertex, say \(\alpha(s)\in PB\).
Then for some small \(\epsilon>0\), \(\alpha(s,s+\epsilon)\) is contained in a path component of
\(\partial B-PB\) of the form \((PB\ast V)-PB\) where $V$ is a
path component of the boundary of the hedge factor of $B$.  If \(\alpha(s,s+\epsilon)\) not contained in the boundary of
a wall, then \(\alpha(s,\overline t)\subset\partial B\).
But \(\alpha|_{(s,\overline t)}\) cannot stay inside \(\partial B\) for \textit{any} B, for if it did,
then \(\alpha(\overline t)\in\partial B\) as well, since block boundaries are closed!
So \(\alpha(s,s+\epsilon)\subset\partial W\) for some wall \(W=B\cap B_+\), and \(\alpha|_{(s,t)}\)
enters $\partial B_+$.  Since this path must also leave $\partial B_+$, it must pass through another
pole, giving us a contradiction.\\

Finally, $RX$ is dense by Lemma \ref{le:denseRX} and
no other safe path component of \(\partial X\) can be dense, since the irrational map provided in Section
\ref{sec:NervesandItineraries} takes components of $IX$ to points of the Cantor set.
\end{proof}

\section{The Proof of Theorem \ref{thm:generalCK}}
The \textit{watermark} of $X$ is defined to be the unordered pair \(\{\mathcal{W}(Y_-),\mathcal{W}(Y_+)\}\).
Let $X_1$ and $X_2$ be two generalized Croke-Kleiner constructions for $G$ from the same spaces
and suppose \(\phi:\partial X_1\to\partial X_2\) is a homeomorphism.
By Proposition \ref{prop:watermark}, there are at least two possible watermarks.
Therefore Theorem \ref{thm:generalCK} will follow from the following proposition.\\

\begin{prop}
\label{prop:watermarx}
$X_1$ and $X_2$ have the homeomorphic watermarks.
\end{prop}

The following are immediate from Theorem D$'$, Corollary \ref{co:localhomology}, and Lemma
\ref{le:fakerightpoles}.\\

\begin{fact}
\(\phi(RX_1)=RX_2\)\\
\end{fact}

\begin{fact}
$\phi$ takes poles to poles.\\
\end{fact}

\begin{lemma}
Let $B_1$ be a block of $X_1$.  Then there is a block $B_2$ of $X_2$ such that:\\
\begin{enumerate}
\item \(\phi(PB_1)=PB_2\).\\
\item \(\phi(\Pi B_1)=\Pi B_2\).\\
\item \(\phi(\partial B_1)=\partial B_2\).\\
\end{enumerate}
\end{lemma}

\begin{proof}
Choose \(\zeta\in PB_1\) and let \(B_2\) be such that \(\phi(\zeta)\in PB_2\).
If the dimension of \(PB_1\) (or \(PB_2\)) is at least 1, then $PB_1$ is connected and since $\phi$
takes poles to poles, we get that \(\phi(PB_1)=PB_2\).  Assume \(PB_1\) (and $PB_2$) both have dimension zero.  Then
they both consist of exactly two points and write \(PB_2=\{\zeta,\eta\}\).
Choose \(\zeta\in\partial H(B_1)\) such that $\{\zeta\}$ is a path component of \(\partial H(B_1)\)
and let \(\alpha:[0,1]\to l(\zeta)\) parameterize the longitude in such a way that 
\(\alpha(0)=\zeta\) to \(\alpha(1)=\eta\).
If \(\phi(\eta)\) is not a pole of $B_2$, then it must be a pole of a neighboring block and \(RX_2-\phi(PB_1)\)
has two path components: \(\phi(\alpha(0,1))\) and \(RX_2-\phi(\alpha[0,1])\).
But \(RX_1-PB_1\) has an infinite number of path components, which gives us a contradiction.
This shows that \(\phi(PB_1)=PB_2\).\\

We get that \(\phi(\Pi B_1)=\Pi B_2\) and \(\phi(\partial B_1)=\partial B_2\) by the following argument:
If \(\alpha:[0,1]\to RX_1\) is a path such that \(\alpha(0)\in PB_1\), \(\alpha(1)\) is a pole, and
\(\alpha(0,1)\) contains no poles, then \(\alpha\subset\partial B_1\) and \(\alpha(1)\) is either a pole of
\(B_1\) or a pole of a neighboring block.
\end{proof}

\begin{lemma}
Let $B_1$ be a block of $X_1$ and $B_2$ be a block of $X_2$
such that \(\phi(PB_1)=PB_2\).
If we have \(\zeta_1\in\partial H(B_1)\) such that \(\{\zeta_1\}\) is a path component of \(\partial H(B_1)\), then
there is a \(\zeta_2\in\partial H(B_2)\) such that \(\phi(l(\zeta_1))=l(\zeta_2)\).
\end{lemma}

\begin{proof}
Whenever \(\{\zeta_i\}\) is a path component of the boundary of the hedge factor of \(B_i\), then
the longitude \(l(\zeta_i)\) is a ball in \(\partial B_i\) whose frontier is precisely \(PB_i\)
and whose interior is a path component of \(\partial B_i-PB_i\).
\end{proof}

Combining these lemmas, we get Proposition \ref{prop:watermarx}.\\

\begin{ack}
The work contained in this paper is one part of the author's Ph.D. thesis written
under the direction of Craig Guilbault at the University of Wisconsin-Milwaukee.  The author would also
like to thank Ric Ancel, Chris Hruska, Boris Okun, and Tim Schroeder for helpful conversations.
\end{ack}

\end{document}